\documentclass[12pt]{amsart}
\usepackage{amssymb}
\usepackage{t1enc}
\usepackage[latin2]{inputenc}
\usepackage{verbatim}
\usepackage{amsmath,amsfonts,amssymb,amsthm}
\usepackage[mathcal]{eucal}
\usepackage{enumerate}
\usepackage[centertags]{amsmath}
\usepackage{graphics}

\newtheorem{theorem}{Theorem}

\numberwithin{equation}{subsection}

\begin{document}
\author{G. Tephnadze}
\title[Strong convergence ]{A note on the strong convergence of
two--dimensional Walsh-Fourier series}
\address{G. Tephnadze, Department of Mathematics, Faculty of Exact and
Natural Sciences, Tbilisi State University, Chavchavadze str. 1, Tbilisi
0128, Georgia}
\email{giorgitephnadze@gmail.com}
\date{}
\maketitle

\begin{abstract}
The main aim of this paper is to investigate the quadratical partial sums of
the two-dimensional Walsh-Fourier series$.$
\end{abstract}



Let $\mathbf{N}_{+}$ denote the set of positive integers, $\mathbf{N:=%
\mathbf{N}_{+}\cup \{}0\mathbf{\}.}$ Denote by $Z_{2}$ the discrete cyclic
group of order 2, that is $Z_{2}=\{0,1\},$ where the group operation is the
modulo 2 addition and every subset is open. The Haar measure on $Z_{2}$ is
given such that the measure of a singleton is 1/2. Let $G$ be the complete
direct product of the countable infinite copies of the compact group $Z_{2}.$
The elements of $G$ are of the form $x=\left(
x_{0},x_{1},...,x_{k},...\right) $ with $x_{k}\in \{0,1\}\left( k\in \mathbf{%
N}\right) .$ The group operation on $G$ is the coordinate-wise addition, the
measure (denote\thinspace $\,$by$\,\,\mu $) and the topology are the product
measure and topology. The compact Abelian group $G$ is called the Walsh
group. A base for the neighborhoods of $G$ can be given in the following
way:
\begin{eqnarray*}
I_{0}\left( x\right) &:&=G, \\
I_{n}\left( x\right) &:&=\,I_{n}\left( x_{0},...,x_{n-1}\right) :=\left\{
y\in G:\,y=\left( x_{0},...,x_{n-1},y_{n},y_{n+1},...\right) \right\} ,
\end{eqnarray*}%
where $x\in G$ and $n\in \mathbf{N}_{+}\mathbf{.}$ Denote $%
I_{n}:=I_{n}\left( 0\right) ,$ for $n\in \mathbf{N.}$

If $n\in \mathbf{N},$ then $n=\sum\limits_{i=0}^{\infty }n_{i}2^{i},$ where $%
n_{i}\in \{0,1\}\,\,\left( i\in \mathbf{N}\right) ,$ i. e. $n$ is expressed
in the number system of base 2. Denote $\left\vert n\right\vert :=\max
\{j\in \mathbf{N:}n_{j}\neq 0\}$, that is, $2^{\left\vert n\right\vert }\leq
n<2^{\left\vert n\right\vert +1}.$

Define the variation of an $n\in \mathbb{N}$ with binary coefficients $%
\left( n_{k},k\in \mathbb{N}\right) $ by

\begin{equation*}
V\left( n\right) =n_{0}+\overset{\infty }{\underset{k=1}{\sum }}\left\vert
n_{k}-n_{k-1}\right\vert .
\end{equation*}

For $k\in \mathbf{N}$ and $x\in G$ let us denote by
\begin{equation*}
r_{k}\left( x\right) :=\left( -1\right) ^{x_{k}}\,\,\,\,\,\,\left( x\in G,%
\text{ }k\in \mathbf{N}\right)
\end{equation*}%
the $k$-th Rademacher function.

The Walsh-Paley system is defined as the sequence of Walsh-Paley functions:
\begin{equation*}
w_{n}\left( x\right) :=\prod\limits_{k=0}^{\infty }\left( r_{k}\left(
x\right) \right) ^{n_{k}}=r_{\left\vert n\right\vert }\left( x\right) \left(
-1\right) ^{\sum\limits_{k=0}^{\left\vert n\right\vert
-1}n_{k}x_{k}}\,\,\,\,\,\,\left( x\in G,\text{ }n\in \mathbf{N}_{+}\right) .
\end{equation*}

The Walsh-Dirichlet kernel is defined by
\begin{equation*}
D_{n}\left( x\right) =\sum\limits_{k=0}^{n-1}w_{k}\left( x\right) .
\end{equation*}

Recall that (see \cite[p.7]{S-W-S})
\begin{equation}
D_{2^{n}}\left( x\right) =\left\{
\begin{array}{c}
2^{n},\text{ \ \ }x\in I_{n} \\
0,\,\,\,\text{\ \ \ }x\notin I_{n}%
\end{array}%
\right.  \label{dir1}
\end{equation}%
and

\begin{equation}
D_{m+2^{l}}\left( x\right) =D_{2^{l}}\left( x\right) +w_{2^{l}}\left(
x\right) D_{m}\left( x\right) ,\text{ when }\,\,m\leq 2^{l}.  \label{dir2}
\end{equation}

Denote by $L_{p}\left( G^{2}\right) ,$ $\left( 0<p<\infty \right) $ the
two-dimensional Lebesgue space, with corresponding norm $\left\Vert \cdot
\right\Vert _{p}.$

The number $\left\Vert D_{n}\right\Vert _{1}$ is called $n$-th Lebesgue
constant. Then (see \cite{S-W-S})

\begin{equation}
\frac{1}{8}V\left( n\right) \leq \left\Vert D_{n}\right\Vert _{1}\leq
V\left( n\right) .  \label{dir3}
\end{equation}

The rectangular partial sums of the two-dimensional Walsh-Fourier series of
a function $f\in L_{1}\left( G^{2}\right) $ are defined as follows:

\begin{equation*}
S_{M,N}f\left( x,y\right) :=\sum\limits_{i=0}^{M-1}\sum\limits_{j=0}^{N-1}%
\widehat{f}\left( i,j\right) w_{i}\left( x\right) w_{j}\left( y\right) ,
\end{equation*}%
where the numbers $\widehat{f}\left( i,j\right) :=\int_{G^{2}}f\left(
x,y\right) w_{i}\left( x\right) w_{j}\left( y\right) d\mu \left( x,y\right) $
\ is said to be the $\left( i,j\right) $-th Walsh-Fourier coefficient of the
function \thinspace $f.$

Let $f\in L_{1}\left( G^{2}\right) $. Then the dyadic maximal function is
given by
\begin{equation*}
f^{\ast }\left( x,y\right) =\sup\limits_{n\in \mathbf{N}}\frac{1}{\mu \left(
I_{n}(x)\times I_{n}(y)\right) }\left\vert \int\limits_{I_{n}(x)\times
I_{n}(y)}f\left( s,t\right) d\mu \left( s,t\right) \right\vert .\,\,
\end{equation*}

The dyadic Hardy space $H_{p}(G^{2})$ $\left( 0<p<\infty \right) $ consists
of all functions for which

\begin{equation*}
\left\Vert f\right\Vert _{H_{p}}:=\left\Vert f^{\ast }\right\Vert
_{p}<\infty .
\end{equation*}

If $f\in L_{1}\left( G^{2}\right) ,$ then (see \cite{Webook2})
\begin{equation}
\left\Vert f\right\Vert _{H_{1}}=\left\Vert \sup\limits_{k\in \mathbf{N}%
}\left\vert S_{2^{k},2^{k}}f\right\vert \right\Vert _{1}.  \label{1}
\end{equation}

It is known \cite[p.125]{G-E-S} that the Walsh-Paley system is not a
Schauder basis in $L_{1}\left( G\right) $. Moreover, there exists a function
in the dyadic Hardy space $H_{1}\left( G\right) $, the partial sums of which
are not bounded in $L_{1}\left( G\right) .$ However, Simon (\cite{Si} and
\cite{si1}) proved that there is an absolute constant $c_{p},$ depending
only on $p,$ such that
\begin{equation}
\frac{1}{\log ^{\left[ p\right] }n}\overset{n}{\underset{k=1}{\sum }}\frac{%
\left\Vert S_{k}f\right\Vert _{p}^{p}}{k^{2-p}}\leq c_{p}\left\Vert
f\right\Vert _{H_{p}}^{p},  \label{2}
\end{equation}%
for all $f\in H_{p}\left( G\right) ,$ where $0<p\leq 1,$ $S_{k}f$ denotes
the $k$-th partial sum of the Walsh-Fourier series of $f$ and $\left[ p%
\right] $ denotes integer part of $p.$ (For the Vilenkin system when $p=1$
see in Gát \cite{gat1}). When $0<p<1$ and $f\in H_{p}\left( G\right) $ the
author \cite{tep2} proved that sequence $\left\{ 1/k^{2-p}\right\}
_{k=1}^{\infty }$ in (\ref{2}) can not be improved.

For the two-dimensional Walsh-Fourier series some strong convergence
theorems are proved in \cite{tep1} and \cite{We}. Convergence of quadratic
partial sums was investigated in details in \cite{GGN, Go}. Goginava and
Gogoladze \cite{gg} proved that the following result is true:

\textbf{Theorem G}. Let $f\in H_{1}\left( G^{2}\right) $. Then there exists
absolute constant $c$, such that
\begin{equation}
\sum\limits_{n=1}^{\infty }\frac{\left\Vert S_{n,n}f\right\Vert _{1}}{n\log
^{2}\left( n+1\right) }\leq c\left\Vert f\right\Vert _{H_{1}}.  \label{3}
\end{equation}

The main aim of this paper is to prove that sequence $\left\{ 1/n\log
^{2}\left( n+1\right) \right\} _{n=1}^{\infty }$ in (\ref{3}) \ is essential
too. In particular, the following is true:

\begin{theorem}
Let \ $\Phi :\mathbf{N}\rightarrow \lbrack 1,$ $\infty )$ be any
nondecreasing function, satisfying the condition $\lim_{n\rightarrow \infty
}\Phi \left( n\right) =+\infty .$ Then%
\begin{equation*}
\underset{\left\Vert f\right\Vert _{H_{1}}\leq 1}{\sup }\underset{n=1}{%
\overset{\infty }{\sum }}\frac{\left\Vert S_{n,n}f\right\Vert _{1}\Phi
\left( n\right) }{n\log ^{2}\left( n+1\right) }=\infty .
\end{equation*}
\end{theorem}

\begin{proof}
Let \qquad
\begin{equation*}
f_{n,n}\left( x,y\right) =\left( D_{2^{n+1}}\left( x\right) -D_{2^{n}}\left(
x\right) \right) \left( D_{2^{n+1}}\left( y\right) -D_{2^{n}}\left( y\right)
\right) .
\end{equation*}%
It is easy to show that%
\begin{equation}
\widehat{f_{n,n}}(i,j)=\left\{
\begin{array}{l}
1,\,\text{if }\left( \,i,\text{\thinspace }j\right) \in \left\{ 2^{n},...,%
\text{ ~}2^{n+1}-1\right\} ^{2},\text{ } \\
0,\text{ if \thinspace }\left( \,i,\text{\thinspace }j\right) \notin \left\{
2^{n},...,\text{ ~}2^{n+1}-1\right\} ^{2}.%
\end{array}%
\right.  \label{4}
\end{equation}

Applying (\ref{dir1}) and (\ref{1}) we have%
\begin{equation}
\left\Vert f_{n_{,}n}\right\Vert _{H_{1}}=\left\Vert \sup\limits_{k\in
\mathbf{N}}\left\vert S_{2^{k},2^{k}}f_{n,n}\right\vert \right\Vert
_{1}=\left\Vert f_{n_{,}n}\right\Vert _{1}=1.  \label{5}
\end{equation}

Let\textbf{\ } $2^{n}<k\leq 2^{n+1}$. Combining (\ref{dir2}) and (\ref{4})
we get

\begin{eqnarray*}
&&S_{k,k}f_{n,n}\left( x,y\right)
=\sum_{i=2^{n}}^{k-1}\sum_{j=2^{n}}^{k-1}w_{i}\left( x\right) w_{j}\left(
y\right) =\left( D_{_{k}}\left( x\right) -D_{2^{n}}\left( x\right) \right)
\left( D_{k}\left( y\right) -D_{2^{n}}\left( y\right) \right) \\
&=&w_{2^{n}}\left( x\right) w_{2^{n}}\left( y\right) D_{k-2^{n}}\left(
x\right) D_{k-2^{n}}\left( y\right) .
\end{eqnarray*}

Using (\ref{dir3}) we have%
\begin{equation}
\left\Vert S_{k,k}f_{n,n}\left( x,y\right) \right\Vert _{1}\geq
\int\limits_{G^{2}}\left\vert D_{k-2^{n}}\left( x\right) D_{k-2^{n}}\left(
y\right) \right\vert d\mu \left( x,y\right) \geq cV^{2}\left( k-2^{n}\right)
.  \label{6}
\end{equation}

Let $\Phi \left( n\right) $ be any nondecreasing, nonnegative function,
satisfying condition $\lim_{n\rightarrow \infty }\Phi \left( n\right)
=\infty .$ Since (see Fine \cite{FF})
\begin{equation*}
\frac{1}{n\log n}\underset{k=1}{\overset{n}{\sum }}V\left( k\right) =\frac{1%
}{4\log 2}+o\left( 1\right) ,
\end{equation*}

using (\ref{5}) and (\ref{6}) and Cauchy-Schwarz inequality we obtain%
\begin{eqnarray*}
&&\underset{\left\Vert f\right\Vert _{H_{1}}\leq 1}{\sup }\underset{k=1}{%
\overset{2^{n+1}}{\sum }}\frac{\left\Vert S_{k,k}f\right\Vert _{1}\Phi
\left( k\right) }{k\log ^{2}\left( k+1\right) }\geq \underset{n=2^{n}+1}{%
\overset{2^{n+1}}{\sum }}\frac{\left\Vert S_{k,k}f_{n,n}\right\Vert _{1}\Phi
\left( k\right) }{k\log ^{2}\left( k+1\right) } \\
&\geq &\frac{c\Phi \left( 2^{n}\right) }{n^{2}2^{n}}\underset{n=2^{n}+1}{%
\overset{2^{n+1}}{\sum }}V^{2}\left( k-2^{n}\right) \geq \frac{c\Phi \left(
2^{n}\right) }{n^{2}2^{n}}\underset{k=1}{\overset{2^{n}}{\sum }}V^{2}\left(
k\right)  \\
&\geq &c\Phi \left( 2^{n}\right) \left( \frac{1}{n2^{n}}\underset{k=1}{%
\overset{2^{n}}{\sum }}V\left( k\right) \right) ^{2}\geq c\Phi \left(
2^{n}\right) \rightarrow \infty ,\text{ when }n\rightarrow \infty .
\end{eqnarray*}

Which complete the proof of Theorem 1.
\end{proof}

\end{document}